\documentclass[12pt,twoside]{amsart}

\usepackage{hyperref}
\usepackage{amsthm, amsmath, amscd, amssymb,centernot}
\usepackage[all]{xy}
\usepackage{palatino,eulervm}
\usepackage[T1]{fontenc}
\usepackage[left=2cm,top=2.5cm,bottom=3cm,right=3cm]{geometry}

\newcommand{\ints}{\mathbb Z}

\newcommand{\str}{\mathcal{O}}

\newcommand{\proj}{\mathbb{P}}

\newcommand{\complex}{\mathbb C}

\theoremstyle{plain}
\numberwithin{equation}{section}
\newtheorem{theorem}{Theorem}[section]
\newtheorem*{theorem*}{Theorem}
\newtheorem{proposition}[theorem]{Proposition}
\newtheorem{lemma}[theorem]{Lemma}

\newtheorem*{conjecture*}{Nagata Conjecture}
\newtheorem*{conjecture1*}{SHGH Conjecture}
\theoremstyle{definition}
\newtheorem{question}[theorem]{Question}

\newtheorem{remark}[theorem]{Remark}
\newtheorem{example}[theorem]{Example}
\newtheorem{discussion}[theorem]{Discussion}

\begin{document}

\title{Positivity of line bundles on general blow ups of $\proj^2$}
\author[Krishna Hanumanthu]{Krishna Hanumanthu}
\subjclass[2010]{Primary 14E25, 14C20, 14H50; Secondary 14E05, 14J26}
\thanks{Author was partially supported by a grant from Infosys Foundation}
\address{Chennai Mathematical Institute, H1 SIPCOT IT Park, Siruseri,
  Kelambakkam 603103, India}
\date{June 21, 2016}
\email{krishna@cmi.ac.in}
\maketitle
\begin{abstract}
Let $X$ be the blow up of $\proj^2$ at $r$ general points
$p_1,\ldots,p_r \in \proj^2$. We study
line bundles on $X$ given by plane curves of degree $d$ passing
through $p_i$ with multiplicity at least $m_i$. 
Motivated by results in \cite{ST}, we establish conditions for ampleness, very ampleness and global
generation of such line bundles.
\end{abstract}

\section{Introduction}
Let $p_1,p_2,\ldots,p_r$ be general points of $\proj^2$. We consider
the blow up $\pi: X \rightarrow \proj^2$ of $\proj^2$ at $p_1,\ldots,p_r$. 
Let $H = \pi^{\star}\str_{\proj^2}(1)$. Let $E_1,E_2,\ldots,E_r$
denote the exceptional divisors of the blow up. 

Line bundles on $X$ of the form 
$L = dH -
\sum_{i=1}^r m_i E_i$ where $d$ is a positive integer and
$m_1,\ldots,m_r$ are non-negative integers are extensively
studied. These line bundles are associated to the linear system of
curves in $\proj^2$ of degree $d$ passing through the point $p_i$ with
multiplicity at least $m_i$, for $1 \le i \le r$. Such linear systems are of great interest and they
are central to many important questions in algebraic geometry.
One important problem concerns the positivity properties of these line
bundles. Specifically, one asks if such line bundles are ample,
globally generated or very ample. There are also various
generalizations of these notions which are studied in this context. 

Several authors have addressed these questions in different situations. 

 Let $L =
dH-\sum_{i=1}^r m_i E_i$. 
First suppose that the points $p_1,\ldots,p_r$ are general.
The ``uniform'' case, where $m_i$ are all equal, has been
well-studied. 
K\"{u}chle \cite{K} and Xu \cite{X2} independently consider the case $m_i = 1$ for all $i$. In this
case, they in fact prove that $L$ is ample if and only if $L^2 > 0$,
when $d \ge 3$. See \cite[Corollary]{K} and \cite[Theorem]{X2}. 
\cite{B} considers the case $m_i = 2$ for all $i$ and gives a similar
criterion (see \cite[Remark, Page 120]{B}). In this situation,
conditions for nefness of $L$ are given in \cite{H2}. 

\cite{BC} gives some technical conditions for
ampleness and very ampleness in general for the blow up of any projective
variety at distinct points.

\cite{Be,DG} treat the case $m_i=1$ for all $i$ for 
{\it arbitrary} points $p_1,\ldots,p_r$.
\cite{Be} gives
conditions for global generation and very ampleness. In 
\cite[Theorem 2]{Be}, it is proved that $L$ is globally generated if $d \ge 3$, $r \le
\frac{\binom{d+5}{2}-4}{3}$ and at most $k(d+3-k)-2$ of the points
$p_i$ lie on a curve of degree $k$, where $1 \le k \le \frac{d+3}{2}$. A
similar result is given for very ampleness in \cite[Theorem
3]{Be}. \cite{DG} also gives conditions for
global generation and very ampleness. 

When the points $p_1,\ldots,p_r$ are general, \cite{AH} considers the
case $m_i=1$ for all $i$ and gives a condition for very
ampleness. It is proved that $L$ is very ample when  $r \le
\frac{d^2+3d}{2}-5$ (\cite[Theorem 2.3]{AH}).  
This result is generalized to arbitrary projective space $\proj^n$ in
\cite[Theorem 1]{C}. A converse to \cite[Theorem 2.3]{AH} is proved and a complete
characterization for very ampleness is given in \cite[Proposition
3.2]{GGP} when $m_i = 1$ for all $i$.

When $p_1,\ldots,p_r$ are arbitrary points on a smooth cubic in $\proj^2$, \cite{H3}
gives conditions for ampleness and very ampleness. The assumption on
the points means that the anticanonical class is effective, making the
variety {\it anticanonical}. In \cite[Theorem
1.1]{H3}, a criterion for ampleness is given in terms of intersection
with a small set of exceptional and nodal classes. In \cite[Theorem
2.1]{H3}, $L$ is proved to be very ample if and only if it is ample
and its restriction to the anticanonical class is very ample. 
These results were generalized in 
\cite{VT} by giving conditions for {\it $k$-very ampleness} of $L$ when
$p_i$ are general points on a smooth cubic. \cite{H5} studies
an anticanonical surface $X$ in general and describes base locus of line
bundles on $X$ and in the process characterizes when these are globally
generated. 

These questions can also be asked when $X$ is a blow up of points on
surfaces other than $\proj^2$. \cite{ST2} studies this question for
$\pi: X \to S$, where $S$ is an abelian surface and $X$ is the blow up
of $S$ at $r$ general points. They obtain conditions for very
ampleness, and more generally for $k$-very ampleness, of
$\pi^{\star}(L)-\sum_{i=1}^r E_i$ where $L$ is a polarization on $S$
and $E_i$ are the exceptional divisors. \cite{ST1} considers this
problem when $S$ is a ruled surface.

In this paper we consider blow ups of $\proj^2$ at general points and
obtain sufficient conditions for ampleness, global generation and very
ampleness. Main motivation for us comes from the work of 
Szemberg and Tutaj-Gasi\'{n}ska  in \cite{ST}. They consider the case 
$2 \le m = m_i$ for all $i$ and show that $L = dH -m \sum_{i=1}^r
E_i$ is ample if $d \ge 3m+1$ and $d^2 \ge (r+1)m^2$ 
(\cite[Theorem 3]{ST}). They also establish conditions for 
$k$-very ampleness of $L$. 
In addition to improving their bounds, we consider the non-uniform
case and give conditions for ampleness and global generation.

Our first main result is Theorem \ref{main} and it gives conditions
for ampleness for an arbitrary line bundle $L = dH-\sum_{i=1}^r
m_iE_i$. 
We then discuss the {\it SHGH Conjecture} (see Section \ref{ample}). Our second main
result is Theorem \ref{maincorollary} which gives better conditions for
ampleness for a uniform line bundle by using 
known cases of the 
SHGH Conjecture. 
Theorem \ref{bpf} gives conditions for global generation of an arbitrary line
bundle and Theorem \ref{bpf-uniform} deals with global generation for
uniform line bundles.
Theorem \ref{veryample} provides conditions for
very ampleness in the uniform case. 

Note that a
necessary condition for ampleness of $L=dH-\sum_{i=1}^r m_iE_i$ is
$L^2 > 0$. This condition is equivalent to $d^2 > \sum_{i=1}^r
m_i^2$. In general this is not sufficient, but for $r \ge 9$, this condition is conjectured to be
sufficient in the uniform case. See the {\it Nagata
Conjecture} in Section \ref{ample}.
In this paper we prove ampleness under the assumption: $d^2 \ge
c_r \sum_{i=1}^r
m_i^2$, where $c_r > 1$ depends on $r$. We give several results with
different values of $c_r$. In Theorem \ref{main}, we give a result for
arbitrary line bundles  
with $c_r = \frac{r+3}{r+2}$. 
In uniform case we get a better bound
$ c_r = \frac{3r+40}{3r+39}$. See Theorem \ref{maincorollary}.  
Though our bounds are not optimal, we do obtain new cases. We compare
our results with some existing results. We also give several examples
to illustrate our results.

The key ingredient in our arguments is a result of Xu \cite[Lemma
1]{X1} and Ein-Lazarsfeld \cite{EL}. This result, using deformation techniques, gives a lower bound for the degree
of a plane curve (over $\complex$) passing through a finite set of general points with
prescribed multiplicities. This bound 
is improved in \cite[Theorem A]{KSS}. 
Ampleness (by Nakai-Moishezon), global
generation and very ampleness (by Reider) can be tested by 
looking at
intersection numbers with effective curves. We bound these intersection
numbers by giving suitable conditions and using the lower bound of Xu
and Ein-Lazarsfeld. 

In Section \ref{ample} we give conditions for ampleness. Conditions
for global
generation and very ampleness are studied in Section \ref{bpf-va}.

We work throughout over the complex number field $\complex$. 
When we say that $p_1,\ldots, p_r$ are {\it general points} of $\proj^2$, 
we mean that they belong to an open dense
subset of $(\proj^2)^r$. More precisely, if a statement holds for 
general points $p_1,\ldots,p_r \in \proj^2$, it holds for all points
in an open dense subset of $(\proj^2)^r$.

{\bf Acknowledgements:}
We sincerely thank Brian Harbourne and Tomasz Szemberg for many useful ideas and
suggestions. Part of this
work was done while the author was visiting the University of
Nebraska-Lincoln. We are grateful 
to the university and the department of mathematics for its
hospitality and stimulating atmosphere.  

\section{Ampleness}\label{ample}
We first give conditions for ampleness of a line bundle on $X$. 

\begin{theorem}\label{main}  Let $p_1,p_2,\ldots,p_r$ be $r \ge 1$ general  points
  of $\mathbb{P}^2$. Consider the blow up $X \to \proj^2$ of
  $\proj^2$ at the points $p_1,\ldots,p_r$. We denote by $H$ the
  pull-back of a line in $\proj^2$ not passing through any of the
  points $p_i$ and by $E_1,\ldots,E_r$ the
  exceptional divisors. 
Let 
$L$ be the line bundle $dH-\sum_{i=1}^r m_iE_i$ for some $d>0$ and
$m_1 \ge \ldots \ge m_r > 0$. 
Then $L$ is ample if 
the following conditions hold: 
\begin{enumerate}
\item $d > m_1+m_2$,
\item $2d > m_1+ m_2 + \ldots + m_5$,
\item $3d > 2m_1+m_2+\ldots+m_7$, and
\item $d^2 \ge \frac{s+3}{s+2} \sum_{i=1}^s
  m_i^2$, for $2 \le s \le r$. 
\end{enumerate}
\end{theorem}
\begin{proof}
We use the Nakai-Moishezon criterion for ampleness. We have 
$L^2 = d^2 - \sum_{i=1}^r m_i^2 > 0$, by hypothesis (4). We show below that $L \cdot C
> 0$ for any irreducible, reduced curve $C$ on $X$. 

Let $C\subset X$ be such a curve. First we consider the case $C =
\sum_{i=1}^r n_iE_i$ for some $n_i \in \ints$, not all zero. If each $n_i$ is
non-positive, then $C$ can not be effective, being negative of an 
effective curve. On the other hand, if $n_i > 0$ for some $i$, then 
$C \cdot E_i = -n_i < 0$. Since $C$ and $E_i$ are both reduced and
irreducible, it follows that $C = E_i$. Then $L \cdot C = m_i > 0$. 

We now assume that $C$ is in the linear
system of a line bundle $eH - \sum_{i=1}^r n_iE_i$ for some positive integer
$e$ and non-negative integers $n_1,\ldots,n_r$ with $n_1 + \ldots + n_r
> 0$. We may further arrange
$n_i$ in decreasing order: 
$n_1 \ge n_2 \ge \ldots \ge n_r$.  
Choose
$s \in \{1,\ldots,r\}$ such that  $n_s \ne 0$ and
$n_{s+1} = \ldots = n_r = 0$. 

Then $C$ is the strict transform of a reduced and irreducible curve of degree
$e$ in $\proj^2$ which passes through $p_i$ with
multiplicity $n_i$, for every $i$.

Then
by \cite[Lemma 1]{X1} or \cite{EL}, we have 
\begin{eqnarray}\label{xu1}
e^2 \ge \sum_{i=1}^r n_i^2 - n_s.
\end{eqnarray} 
Along with hypothesis (4) of the theorem, this gives 
\begin{eqnarray}\label{xu}
d^2 e^2 \ge \frac{r+3}{r+2}\left(\sum_{i=1}^r m_i^2\right)\left(\sum_{i=1}^r n_i^2 -
n_s\right).
\end{eqnarray}

We now consider different cases. 

{\bf Case A:} Let $n_1 \ge 2$. 

We first suppose that $s=r$.

If $r=2$ and $n_1=n_2=2$, then $e \ge 3$. So hypothesis (1) gives $L \cdot C >
0$. We assume henceforth that if $r = 2$, then $(n_1,n_2) \ne (2,2)$. 

Then, 
by \eqref{xu} and Lemma \ref{key} below, $d^2e^2 \ge {(\sum_{i=1}^r m_i n_i)}^2$. Hence 
$de \ge \sum_{i=1}^r m_i n_i$ and $L \cdot C \ge 0$. 
We also know that the inequality in
Lemma \ref{key} is strict except in two cases which we treat below. In
addition, if $de =  \sum_{i=1}^r
m_i n_i$, we must have an equality in \eqref{xu1}. So $e^2 = \sum_{i=1}^r n_i^2 - n_r$.

If $(n_1,\ldots,n_r) = (2,1,\ldots,1)$ then $e^2 = r+2$. If $r=2$ then
$C = 2H-2E_1-E_2$ and $L \cdot C > 0$ by hypothesis (1). If $r=7$,
then $C = 3H-2E_1-E_2-\ldots-E_7$ and $L\cdot C > 0$ by hypothesis
(3). For $r=14$, there is no curve of degree 4 passing through $p_1$
with multiplicity 2 and through $p_2,\ldots,p_{14}$ with multiplicity
1. For, a point of multiplicity 2 imposes 3 conditions and
13 simple general points impose another 13 conditions. So
there are a total of 16 conditions, while the space of plane curves of
degree 4 has only dimension $\binom{6}{2} - 1 = 14$. Same argument
holds when $r > 14$. 

We get an equality in Lemma \ref{key} also when $r=3$ and
$n_1=n_2=n_3=2$. But in this case $L \cdot C > 0$ by hypothesis
(1). Indeed, by (1), $2d > 2m_1 + 2m_2 \ge 4m_3 \Longrightarrow d > 2m_3
\Longrightarrow
3d>2m_1+2m_2+2m_3$.

Thus in all cases $L \cdot C > 0$ as desired.

Next suppose that $s < r$. If $s=1$ then $e \ge n_1$. Since $d > m_1$, by
hypothesis (1), we get $de > m_1n_1$.  

Let $s \ge 2$. 
By hypothesis (4), 
$d^2 \ge \frac{s+3}{s+2} \sum_{i=1}^s m_i^2$. 
So as above (applying Lemma \ref{key} with $r=s$), it
follows that $de > \sum_{i=1}^s m_i n_i$. Thus 
$L \cdot C > 0$.

{\bf Case B:} Let $C = eH - E_1 -E_2 - \ldots -E_s$. 

We have $e^2 \ge s-1$. For $s \le 5$, hypotheses (1) and
(2) show that $L \cdot C > 0$. Let $s \ge 6$, so that $e \ge 3$.
We claim that $e^2 \ge s$. Indeed, the space of plane curves of
degree $e$ has dimension $\binom{e+2}{2}-1$ and each point $p_i$
imposes one condition. So $\binom{e+2}{2}-1 \ge s$. It is easy to
see that $e^2 \ge \binom{e+2}{2}-1$ when $e \ge 3$, so that $e^2 \ge
s$ as claimed.

By hypothesis (4), $d^2 > \sum_{i=1}^s m_i^2$. The following is a standard inequality:
$$s(m_1^2+\ldots+m_s^2) \ge {(m_1+\ldots+m_s)}^2.$$ So
$d^2  >  \frac{{(\sum_{i=1}^s
  m_i)}^2}{s}$. Hence ${(de)}^2 \ge d^2s >  {\left(\sum_{i=1}^s m_i\right)}^2$. 

 We conclude that $L$ is ample.  
\end{proof}
\begin{remark}\label{necessary}
The hypotheses (1)-(3) of Theorem \ref{main} are necessary in order
for $L$ to be ample. This is due to the existence, respectively, of a line
through two general points, a conic through five general points and a
cubic through seven general points with one of the points being a
double point and all others simple points.
\end{remark}
\begin{lemma}\label{key}
Let $r \ge 2$, $m_1, m_2 \ldots, m_r \ge 0$ and $n_1 \ge n_2 \ge
\ldots \ge 
n_r > 0$. Suppose that 
\begin{itemize}
\item $n_1 \ge 2$, and 
\item if $r=2$ then $(n_1,n_2) \ne (2,2)$.
\end{itemize} 
Then the following inequality holds:
$$ \frac{r+3}{r+2} \left(\sum_{i=1}^r m_i^2\right)\left(\sum_{i=1}^r n_i^2 - n_r\right) \ge
{\left(\sum_{i=1}^r m_i n_i\right)}^2.$$ Moreover, the inequality is strict when
$(n_1,n_2,\ldots,n_r) \ne (2,1,\ldots,1)$ and if $r=3$, $(n_1,n_2,n_3)
\ne (2,2,2)$. 
\end{lemma}
\begin{proof}
For simplicity, let $a = \sum_{i=1}^r m_i^2, b = \sum_{i=1}^r n_i^2$ and $c =
\sum_{i=1}^r m_i n_i$. So the desired inequality is $\frac{r+3}{r+2}a(b-n_r) \ge c^2$.

Rearranging terms and clearing the denominator, the desired inequality
is equivalent to 
$ (r+3)ab -(r+2)c^2 \ge (r+3)n_ra,$ which in turn is equivalent to 
\begin{eqnarray}\label{desired}
ab + (r+2)(ab-c^2) \ge (r+3)n_ra.
\end{eqnarray}

Now we apply the well-known equality: 
\begin{eqnarray}\label{cauchy}
ab-c^2 =  \left( \sum_{i=1}^r m_i^2\right) \left( \sum_{i=1}^r
  n_i^2\right) - {\left( \sum_{i=1}^r m_i n_i \right)}^2 = \sum_{i < j}
{\left(m_in_j - m_jn_i\right)}^2.
\end{eqnarray}

The left hand side of \eqref{desired} is then $ab
+ (r+2)  \sum_{i < j}
{(m_in_j - m_jn_i)}^2$ and it is at least $ab$.

We now show the following inequality from which \eqref{desired} and the lemma follow.
\begin{eqnarray}\label{final-ineq}
b \ge (r+3)n_r.
\end{eqnarray}
If $n_r = 1$, then the right hand side of \eqref{final-ineq} is
$r+3$. When $n_r=1$ the
left hand side of \eqref{final-ineq} is smallest when $(n_1,\ldots,n_r)=(2,1,\ldots,1)$ and
the minimum in this case is $r+3$. Suppose now that $n_r > 1$.
Then minimum of the left hand side of \eqref{final-ineq} is obtained when
$n_1=n_2=\ldots=n_r$ and it is $rn_r^2$. To finish, note that  
$rn_r^2 \ge (r+3)n_r$ is equivalent to $r(n_r-1) \ge 3$. This fails
only when $r=2$ and $n_r = 2$ which we omitted in the
lemma. Further it is an equality only when $r=3$ and $n_r=2$. In all
other cases we have a strict inequality and the lemma follows. 
\end{proof}

\begin{example}\label{optimal}
Let $L_d = dH-3E_1 - \sum_{i=2}^82E_i -\sum_{i=9}^{12}
  E_i$. Then the smallest $d$ which satisfies the hypotheses (1)-(4)
  of Theorem \ref{main} is $d = 7$: $d$ must satisfy $d > 5$, $2d >
  11$, and $3d > 18$. It is clear that $d=7$ satisfies
  hypothesis (4). For instance $7^2 > \frac{15}{14}(41) = 43.92$. 
 So  $L_7 = 7H-3E_1 - \sum_{i=2}^82E_i -\sum_{i=9}^{12}
  E_i$ is ample. If  $d < 7$, $L_d^2 < 0$ and thus $L_d$ is not ample. So we get the optimal
  result in this example. 
\end{example}
\begin{example}\label{not-optimal} Let $L_d = dH-3E_1 - \sum_{i=2}^{10}2E_i -\sum_{i=11}^{12}
  E_i$. Then the smallest $d$ which satisfies the hypotheses (1)-(4)
  of Theorem \ref{main} is $d = 8$. By hypothesis (4), we require $d^2
  \ge \frac{15}{14}(47) = 50.35$. It is easy to see that $d=8$
  satisfies all the other hypotheses. 
So $L_8 = 8H-3E_1 - \sum_{i=2}^{10}2E_i -\sum_{i=11}^{12}
  E_i$ is ample. But $L_7^2 > 0$ and it is ample if the SHGH
  Conjecture, which is recalled below, is
  true. See also Example \ref{not-optimal1}. 
Our result in this case is thus not expected to be 
  optimal. 
\end{example}

\begin{example}\label{tight}
Consider  $L_d = dH - 10 \sum_{i=1}^5 E_i$. Then $L_{25}$ is not
ample, because there is a conic through 5 general points and 
$L_{25} \cdot (2H-\sum_{i=1}^5 E_i) = 0$. Note that $L_{25}^2 > 0$, 
showing that $L^2 > 0$ is not a sufficient
condition for ampleness when $r =5$. 
On the other hand, $L_{26}$ is ample by Theorem \ref{main}.  
\end{example}

\begin{question}\label{question}
Does the conclusion of Theorem \ref{main} hold with the hypotheses
(1)-(3) and 
only that $ d^2 \ge \frac{r+3}{r+2}\sum_{i=1}^r m_i^2$ (that is, we only
require the hypothesis (4) for $s=r$)? Our calculations
suggest that the answer is yes, but we do not have a proof. 
\end{question}

However, we have the following theorem which says that hypothesis (4)
need only be checked for $s \ge 9$, if we require that $L$ meets {\it
  exceptional curves} positively. 

Let $C$ be a curve on $X$. We say that $C$ is an 
{\it exceptional curve} or a  {\it $(-1)$-curve} if 
it is a
smooth rational curve such that $C^2 = -1$. 
If $C$ is a $(-1)$-curve, by adjunction, we also have $C \cdot K_X =
-1$, where $K_X = -3H+\sum_{i=1}^r E_i$ is the canonical line bundle
on $X$.

\begin{proposition} \label{main1}  Let $p_1,p_2,\ldots,p_r$ be $r \ge 9$ general  points
  of $\mathbb{P}^2$. Consider the blow up $X \to \proj^2$ of
  $\proj^2$ at the points $p_1,\ldots,p_r$. We denote by $H$ the
  pull-back of a line in $\proj^2$ and by $E_1,\ldots,E_r$ the
  exceptional divisors. 
Let 
$L$ be the line bundle on $X$ given by $dH-\sum_{i=1}^r m_iE_i$ for $d>0$ and
$m_1 \ge \ldots \ge m_r > 0$. 
Then $L$ is ample if 
the following conditions hold: 
\begin{enumerate}
\item $d > m_1+m_2$,
\item $2d > m_1+ m_2 + \ldots + m_5$,
\item $3d > 2m_1+m_2+\ldots+m_7$, 
\item $d^2 \ge \frac{s+3}{s+2} \sum_{i=1}^s m_i^2$, for $9 \le s \le r$, and
\item $L \cdot E > 0$, where $E$ is one of the following 
    exceptional curves\\
$4H-2E_1-2E_2-2E_3-E_4-...-E_8$, or\\
$5H-2E_1-...-2E_6-E_7-E_8$, or\\
$6H-3E_1-2E_2-...-2E_8$.
\end{enumerate}
\end{proposition}
\begin{proof}
The proof is exactly as in Theorem \ref{main}. The difference here is
that hypothesis (4) is required for only $s \ge 9$. For $s < 9$, we use
the following fact. 

Let $C$ be a curve in $X$ which belongs to the linear system of a
divisor of the form $eH-n_1E_1-\ldots-n_sE_s$ with 
$s<9$, $e>0$ and each $n_i \ge 0$. The
blow up $X \rightarrow \proj^2$ factors as $X \to X_1 \to \proj^2$ where $X_1$ is the
blow up of $\proj^2$ at $p_1,\ldots,p_s$. Then $C$ is isomorphic to
its image in $X_1$. Denote this image also by $C$.  Then
the class of $C$ in the Picard lattice of $X_1$ is a
linear combination of a numerically effective class and a finite set
of exceptional classes. See \cite[Theorem 1(a)]{H4} for a proof of
  this fact. Hence the class of $C$ in the Picard lattice of $X$ is
  also such a linear combination.

Moreover there are only finitely many exceptional classes 
of curves coming from plane curves passing through less than 9 points.  
The proof of \cite[Theorem 4a, Page 284]{N2} lists these
  exceptional classes (see also Discussion \ref{exccurves}). 
Of the seven classes listed there three
  correspond to the classes in hypothesis (5). Hypotheses (1)-(3) and
  the hypothesis $m_i >0$ account for the other exceptional
  curves in the list.

Now if an irreducible and reduced curve $C$ is numerically effective
then by Lemma \ref{uniform} below it follows that $L \cdot C > 0$. If
$C$ is an exceptional curve then $L \cdot C > 0$ by hypothesis. 
\end{proof}

\begin{remark}\label{x3}
In \cite[Theorem 2]{X3}, Xu proves a result comparable to our
Theorem \ref{main}. In addition to some necessary conditions to ensure
$L$ meets exceptional curves positively, the main hypothesis for
\cite[Theorem 2]{X3} is $d^2 >
\frac{10}{9}\sum_{i=1}^r m_i^2$. This result follows from Proposition
\ref{main1} because $\frac{s+3}{s+2} \le \frac{10}{9}$ when $s \ge 7$. 
\end{remark}

We make the following useful observation. 

\begin{lemma}\label{uniform}
Let $L = dH-\sum_{i=1}^r m_iE_i$ be a line bundle on $X$ such that $L^2 >
0$. Let $C$ be an irreducible and reduced curve on $X$ with $C^2 \ge 0$. 
Then $L \cdot C > 0$. \end{lemma}
\begin{proof}
Write $C = eH-\sum_{i=1}^r n_i E_i$. We may assume that $e > 0$ and
$n_i  \ge 0$ for every $i$. 
Since $C^2 \ge 0$, we have $e^2 \ge \sum_{i=1}^r n_i^2$. 
Since $L^2 > 0$, we have $d^2 > \sum_{i=1}^r m_i^2$.
Thus $d^2e^2 > \left(\sum_{i=1}^r m_i^2\right) \left(\sum_{i=1}^r n_i^2\right) \ge
{\left(\sum_{i=1}^r m_in_i\right)}^2$. The last inequality follows from
\eqref{cauchy}. 
Thus $L \cdot C > 0$. 
\end{proof}

Thus a line bundle $L$ with $L^2 > 0$ fails to be ample only if
there is a curve $C$ such that $C^2 < 0$ and $L
\cdot C < 0$.

When $m_i = m$ for all $i$ (uniform case), Question \ref{question} has a positive
answer. Hypothesis (4) of Theorem \ref{main} 
is then simply $d^2 \ge  \frac{r(r+3)}{r+2}m^2$. Indeed, since 
$\frac{r+3}{r+2} \ge
\frac{s+3}{s+2}$ for any $s \le r$, it automatically follows 
that $d^2 \ge \frac{s(s+3)}{s+2}m$ for any $2 \le s \le r$. 

In fact, we have the following corollary in the
uniform case. 

For $r \ge 2$, we define $\lambda_r$ as follows: 
$\lambda_2 = \lambda_3 = 2$, $\lambda_5 = 2.5$ and for all other $r$, 
$\lambda_r = \sqrt{\frac{r(r+3)}{r+2}}$. 

\begin{proposition}\label{key-cor} Let $r \ge 2$. 
Let $L = dH - m \sum_{i=1}^r E_i$ be a line bundle on $X$ with $m >
0$. Suppose that $d >
\lambda_r m$. Then $L$ is ample. 
\end{proposition}
\begin{proof}
 Let $C = eH-\sum_{i=1}^r n_i E_i$ be an irreducible, reduced
curve on $X$. We show that $L \cdot C > 0$.

We may assume that $e > 0$. Write $n_1 \ge n_2 \ge \ldots \ge n_r$. 
We assume without loss of generality that $n_r > 0$. 

When $n_1=1$ it
is clear that $L \cdot C > 0$, as in {\bf Case B} in the proof of
Theorem \ref{main}. So we assume
$n_1 \ge 2$. 

For $r=2,3,5$ the bound in hypothesis (4) of Theorem \ref{main} is
lower than the necessary conditions contained in the hypotheses (1)-(3) of
that theorem. These numbers are respectively, 2, 2 and 2.5 for
$r=2,3,5$, which we defined to be $\lambda_2,\lambda_3$ and
$\lambda_5$. So we have ampleness in these cases. 

Now we apply Lemma \ref{key} with $m_i=m$ for all $i$.
Thus  
$$ \frac{(r+3)r}{r+2} m^2\left(\sum_{i=1}^r n_i^2 - n_r\right) \ge
m^2 {\left(\sum_{i=1}^rn_i\right)}^2.$$

Rearranging terms
we get $\sum_{i=1}^r n_i^2 - n_r \ge  \frac{r+2}{r(r+3)}
{\left(\sum_{i=1}^rn_i\right)}^2$. So $e \ge
\sqrt{\frac{r+2}{r(r+3)}}(\sum_{i=1}^rn_i) =
\frac{1}{\lambda_r}(\sum_{i=1}^rn_i)  $. 
Thus $de > m \sum_{i=1}^r n_i$.
\end{proof}

\begin{example}\label{imp}
Consider $r=8$ and the line bundle $L_d = dH-60\sum_{i=1}^8E_i$. Note
that $L_{169}^2 = 28561-28800 = -239 < 0$ and $L_{170}^2 = 28900-28800 > 0$.

By \cite{N2} (see the Introduction to \cite{H1}) there
is a degree 48 curve in $\proj^2$ passing through 8 general points with
multiplicity 17 at each of the 8 points. In the notation of \cite{H1},
$\delta(m,n)$ is the least integer such that there is a curve of that
degree in $\proj^2$ passing through $n$ general points with multiplicity
at least $m$. For $n \le 9$, $\delta(m,n)$ is the ceiling of $c_n m$
where $c_n$ is an explicitly determined rational number. 
It turns out that $c_8 = \frac{48}{17}$ by \cite{N2}. 

The strict transform of such a curve on the blow up of the 8 points is
a curve $C$ in the class of $48H-17\sum_{i=1}^8 E_i$. Then 
$(170H-60\sum_{i=1}^8E_i)\cdot (48H-17\sum_{i=1}^8 E_i) = 0$. 
So $L_{170}$ is not ample. An exceptional curve
in the base locus of the linear system of $C$ is $E =
6H-3E_1-2E_2-\ldots-2E_8$. It is easily checked
that $L_{170}\cdot E = 0$. We can verify that $E$ is indeed an
exceptional curve by using Nagata's result that exceptional curves form a
single orbit under a group action and the fact that $E$ is a
translate of $E_8$ under this group action. 
See Discussion \ref{exccurves} and the
proof of Lemma \ref{bound-on-e}. 


We see that $L_{178}$ is ample by Proposition \ref{key-cor}. For, $\lambda_8 = 2.9664$ and
$\lambda_8m = 177.9887$. We consider this
example again in Example \ref{imp1} and show in fact that $L_{171}$ is
ample. 
\end{example}

Our goal now is to obtain better bounds in the uniform
case.  We first recall the so-called {\it SHGH Conjecture}.

Let $p_1,\ldots,p_r$ be general points in $\proj^2$ and let
$d,m_1,\ldots,m_r$ be non-negative integers. Let $L$ be the linear system of degree $d$ curves in $\proj^2$ passing
through $p_i$ with multiplicity at least $m_i$. Then $L$ corresponds
to a line bundle $dH-\sum_{i=1}^r m_iE_i$ on $X$, where $X$ is the blow up
of $\proj^2$ at the points $p_i$. We denote this line bundle also by
$L$. 

Note that the projective space of
plane curves of degree $d$ has dimension $\binom{d+2}{2}-1$ and a
point of multiplicity $m_i$ imposes $\binom{m_i+1}{2}$ conditions.
So we say that 
the {\it expected dimension} of the linear system $L$ is
$\binom{d+2}{2}-1-\sum_{i=1}^r \binom{m_i+1}{2}$ if this number is
non-negative and -1 otherwise. 
The
actual dimension of $L$ is greater than or equal to the expected
dimension. The linear system $L$ is called {\it special} if its actual
dimension is more than the expected dimension.

By Riemann-Roch theorem $L$ is special if
and only if 
$h^0(X,L) h^1(X,L) \ne  0$. 


Easy examples of special linear systems are given by
$(-1)$-curves. If $C$ is a $(-1)$-curve, then the linear system of
$2C$ is special. Indeed, by Riemann-Roch, we have 
$h^0(2C) -h^1(2C) = \frac{2C\cdot(2C-K)}{2} + 1$. Note that $h^2(2C) =
0$ since $2C$ is effective. Also, $2C\cdot(2C-K) = 4C^2 - 2C\cdot K = -4
+2 = -2$. So $h^0(2C) -h^1(2C) = 0$, which implies that $h^0(2C)$ and
$h^1(2C)$ are both nonzero. In fact, they are both equal to 1. 
More generally, any linear system with a  
$(-1)$-curve in its base locus (with multiplicity at least 2)
is special. 

The SHGH conjecture predicts that every special linear
system arises from $(-1)$-curves as above.  
This conjecture was formulated by Segre \cite{S}, Harbourne \cite{H8}, Gimigliano
\cite{G} and Hirschowitz \cite{Hi}. There are different formulations of
this conjecture. See 
\cite{H7} for a nice survey. \cite{CM3} shows equivalence of various
versions. 

The following is one version of the SHGH Conjecture.
\begin{conjecture1*}
Let $X$ be the blow up of $\proj^2$ at $r$ general points. Then the
following statements hold. 
\begin{enumerate}
\item Any reduced, irreducible curve on $X$ with negative self-intersection is a
(-1)-curve;
\item Any nef and effective linear system is non-special. 
\end{enumerate}
\end{conjecture1*}

See \cite{CM1,CM2,M,dF,Y,DJ,CHMR} for some progress on this conjecture. In
particular, \cite[Theorem 34]{DJ} verifies the SHGH Conjecture 
when 
$m_i \le 11$ for all $i$. We also note that
\cite[Theorem 2.5]{dF} verifies Statement (1) of the SHGH Conjecture
when one of the multiplicities is 2, that is, $m_i = 2$ for some $i$. 

\begin{discussion}\label{exccurves}
Let $r \ge 3$.
For the divisor class group $Cl(X)$ of $X$, consider the 
linear map $\gamma_0: Cl(X) \to Cl(X)$, given by:  

$\gamma_0(H) = 2H-E_1-E_2-E_3$, 

$\gamma_0(E_i) = H-E_1-E_2-E_3+E_i$ for $i=1,2,3$, and 

$\gamma_0(E_i) = E_i$ for $i \ge 4$.  

For $i=1,2,\ldots,r-1$, let $\gamma_i: Cl(X) \to Cl(X)$ be the map
which interchanges $E_i$ and $E_{i+1}$ and fixes  $H$ and $E_j$ when $j
\notin \{i,i+1\}$.

Let $W_r$ denote the group of linear
automorphisms of $Cl(X)$ generated by
$\gamma_0,\gamma_1,\ldots,\gamma_{r-1}$. There is a root system on
$Cl(X)$ with simple roots given by 
$$s_0 = H-E_1-E_2-E_3, s_i = E_i-E_{i+1}, {\textrm ~for~} 1 \le i \le r-1.$$ 

The maps $\gamma_i$ may be regarded as reflections orthogonal to the
simple roots. 
More concretely, for any $L \in Cl(X)$ and $i$, $\gamma_i(L) = L +
(L\cdot s_i)s_i$. In this point of view, $W_r$ is the Weyl group of the root
system on $Cl(X)$ with simple roots $s_0,\ldots,s_{r-1}$.
See
\cite{L} and \cite{H9} for more details. 
\end{discussion}

\begin{lemma}\label{bound-on-e}
Let $r \ge 3$. Let $C$ be an exceptional curve on $X$ given by $eH-\sum_{i=1}^r
n_iE_i$ with $e > 0$ and $n_1 \ge n_2 \ge \ldots \ge n_r$.
Then $e < n_1+n_2+n_3$. 
\end{lemma}
\begin{proof}
Nagata \cite{N2} proved that exceptional curves on $X$ form a single orbit under
the action of $W_r$ on Cl$(X)$, when $r \ge 3$. 
We consider the fundamental domain for the action of $W_r$ on
Cl$(X)$. This consists of line bundles with non-negative
intersection with all simple roots. Since the exceptional curves form
a single $W_r$-orbit, exactly one exceptional curve is in the
fundamental domain. It is
clear that $E_r$ is in the
fundamental domain, since it has
non-negative intersection with all simple roots: $E_r \cdot s_i \ge 0$
for $i= 1,\ldots, r$. 

Since $C \ne E_r$, it is not in the fundamental domain. So $C\cdot s_i < 0$
for some simple root $s_i$. 
Since $n_1 \ge \ldots \ge n_r$, $C$ meets the roots $s_1,\ldots, s_r$ non-negatively. Hence
$C \cdot s_0 = e-n_1-n_2-n_3<0$. 
\end{proof}

\begin{proposition}\label{shgh-nagata}
Let $L = dH-m\sum_{i=1}^r E_i$ be a line bundle on $X$ with $L^2 >
0$. Suppose that there exists a positive integer $N$ such that the SHGH Conjecture holds for all curves of the form 
$eH -\sum_{i=1}^r n_i E_i$ with $0 \le n_i \le N$ for all $i$. 
Suppose further that $L$ meets positively all the finitely many exceptional
curves with multiplicity at most $N$ at $p_i$ for all $i$. 

Now let $C = eH - \sum_{i=1}^r n_i E_i$ with $0 \le n_i \le N$ be an
irreducible and reduced curve in $X$. Then $L \cdot C > 0$. 
\end{proposition}
\begin{proof}
If $C^2 < 0$, then the SHGH
Conjecture implies that $C$ is a $(-1)$-curve. So $L \cdot C > 0$ by
hypothesis. 
If $C^2 \ge 0$, we are done by Lemma \ref{uniform}. 
\end{proof}

\begin{remark}\label{mult7}
The SHGH Conjecture is known to be true for curves with multiplicities up
to 11, by \cite[Theorem 34]{DJ}. So the result of Proposition \ref{shgh-nagata} holds with $N=11$.
\end{remark}

Next we will prove a lemma similar to Lemma \ref{key}. 

\begin{lemma}\label{key2}
Let $r \ge 9$ and $n_1 \ge n_2 \ge
\ldots \ge 
n_r \ge  3$. Suppose that $n_1 > 11$. 
Then the following inequality holds:
\begin{eqnarray}\label{ineq}
\frac{(3r+40)r}{3r+39}\left(\sum_{i=1}^r n_i^2 - n_r\right) >
{\left(\sum_{i=1}^r n_i\right)}^2.
\end{eqnarray}
\end{lemma} 
\begin{proof}
This is similar to the proof of Lemma \ref{key}. 
Let $a = \sum_{i=1}^r n_i^2$ and $b = \sum_{i=1}^r n_i$.
Clearing the denominator and rearranging terms,
the desired inequality is $ra+(3r+39)(ra-b^2) > r(3r+40)n_r$. Since 
$ra - b^2 \ge 0$ the lemma  will follow if $ra > r(3r+40)n_r$, or equivalently, if
$a > (3r+40)n_r$. 

For $n_r=3$, the smallest value of $a$ is obtained when
$(n_1,\ldots,n_r) = (12,3,\ldots,3)$ and then $a = 144+9r-9 = 9r+135$
which is clearly greater than  $3(3r+40)$. When
$n_r = 4$, the smallest value of $a$ is $16r+128$ and it is obtained when 
$(n_1,\ldots,n_r) = (12,4,\ldots,4)$. It is easy to see that we have $16r+128 > 4(3r+40)$ for $r
\ge 9$. When $n_r = 5$, the smallest value of $a$ is $25r+119$ attained when 
$(n_1,\ldots,n_r) = (12,5,\ldots,5)$ and again it is easy to verify
that $25r+119 > 5(3r+40)$ for $r \ge 9$. In
exactly the same way, the required inequality follows when 
$n_r \ge
6$.
\end{proof}

Now we are ready to prove our main result about ampleness in the uniform case. 

\begin{theorem}\label{maincorollary}
Let $p_1,p_2,\ldots,p_r$ be $r \ge 1$ general  points
  of $\mathbb{P}^2$. Consider the blow up $X \to \proj^2$ of
  $\proj^2$ at the points $p_1,\ldots,p_r$. We denote by $H$ the
  pull-back of a line in $\proj^2$ not passing through any of the
  points $p_i$ and by $E_1,\ldots,E_r$ the
  exceptional divisors. 
Let $L = dH - m \sum_{i=1}^r E_i$ be a line bundle on $X$ with $m > 0$. Then $L$ is
ample if
\begin{enumerate}
\item $d > \frac{95}{32} m$, and
\item $d^2 \ge \big{(}\frac{3r+40}{3r+39}\big{)}rm^2$. 
\end{enumerate}
\end{theorem}
\begin{proof}
As in the proof of Theorem \ref{main} we use the Nakai-Moishezon criterion. Let $C =
eH - \sum_{i=1}^r n_iE_i$ be an irreducible, reduced curve on $X$. 
We consider different cases.

{\bf Case 1:} First suppose that $0 \le n_i \le 11$ for all $i$. 

If $C^2 \ge 0$, then Lemma \ref{uniform} implies that $L \cdot C > 0$. 

If $C^2 < 0$, then $C$ is a $(-1)$-curve. See Remark \ref{mult7}. By
Lemma \ref{bound-on-e}, we have $e \le 32$. So $96e-95e \le 32
\Rightarrow 3e-1
\le \frac{95}{32}e$.
Further, since $K_X \cdot
C = -1$, we get $\sum_{i=1}^r n_i = 3e-1$. It follows now that $L \cdot C >
0$. For, by hypothesis (1), $de > \frac{95}{32}me \ge m(3e-1) =
m\sum_{i=1}^r n_i$.  

We remark that we do have several exceptional curves 
with $e=32$. For instance, consider  
the linear system
$D = 32H-15E_1-10\sum_{i=2}^9 E_i$.
It is possible
to reduce $D$ to $E_9$ by applying elements of $W_9$ (see
Discussion \ref{exccurves}). As we noted in the proof of Lemma
\ref{bound-on-e}, Nagata showed that exceptional curves form a single
orbit under the action of $W_9$. So it follows that $D$ is a
$(-1)$-curve. 

{\bf Case 2:} We assume now that $n_1 \ge 12$. Write $n_1 \ge n_2
\ge \ldots \ge n_r$. Without loss of
generality, we assume that $n_r \ne 0$.

{\bf Case  2(i):} Suppose that $n_r=1$.

By \cite[Lemma 1]{X1}, $e^2 \ge n_1^2+n_2^2+\ldots+n_{r-1}^2$. Thus 
$C^2 = e^2 - n_1^2 - n_2^2 - \ldots n_{r-1}^2 -1 \ge -1$. If $C^2 \ge
0$, then we 
are done by Lemma \ref{uniform}.
So suppose that $C^2 = -1$. 

Let $a = \sum_{i=1}^r n_i^2$ and $b = \sum_{i=1}^r
n_i$. Then $e^2 = a-1$.

We claim that $ra - b^2 \ge r$. Indeed, we have $ra-b^2 = \sum_{i,j}
(n_i-n_j)^2$. Since we have $n_1 \ge 12$ and $n_r =1$, the number of
non-zero terms in the sum $\sum_{i,j} (n_i-n_j)^2$  is at least $r-1$. In fact, the number of
non-zero terms is at least $r$ unless the $(n_1,\ldots,n_r) =
(12,\ldots,12,1)$ or $(n_1,\ldots,n_r) =
(12,1,\ldots,1)$. In both cases it is clear 
that the sum $\sum_{i,j} (n_i-n_j)^2$ is at least $r$. 

Now $ra-r \ge b^2 \Rightarrow a-1 \ge \frac{b^2}{r}$. Thus $d^2e^2 >
(a-1) rm^2 \ge r\frac{b^2}{r}m^2 = b^2m^2$ and $de > bm$.

{\bf Case 2(ii)}: Suppose that $n_r =2$. 

By \cite[Theorem 2.5]{dF}, any irreducible, reduced curve which passes
through one of the points $p_i$ with multiplicity 2 is a $(-1)$-curve
if it has negative self-intersection. 
Thus if $C^2 < 0$ then  
$C$ is a $(-1)$-curve. In particular, $C^2 = e^2-\sum_{i=1}^r n_i^2 = -1$. 
Exactly as in {\bf Case 2(i)} we conclude $L \cdot C > 0$. If $C^2 \ge
0$, we are done by Lemma \ref{uniform}.

{\bf Case 2(iii)}: $n_r \ge 3$. 

Suppose first that $r \le 8$.
According to hypothesis (1), $d >
\left(\frac{95}{32}\right)m = 2.96875 m$. In the
notation of Proposition \ref{key-cor}, 
$\lambda_8 =  2.966$ and $\lambda_r \le \lambda_8$ 
for $r \le 8$.  So $L$ is ample by Proposition \ref{key-cor}.
 
When $r \ge 9$, we use Lemma \ref{key2}. We have, using \eqref{xu1}, 
\begin{eqnarray*}
&& e^2\ge \sum_{i=1}^r n_i^2 - n_r > \frac{3r+40}{r(3r+39)}{\left(\sum_{i=1}^r
  n_i\right)}^2\\ &&\Longrightarrow d^2e^2 > m^2 {\left(\sum_{i=1}^r
  n_i\right)}^2  \Longrightarrow de > m \left(\sum_{i=1}^r
  n_i\right)
\Longrightarrow L \cdot C > 0.
\end{eqnarray*}
We conclude that $L$ is ample. 
\end{proof}

\begin{remark}\label{analysis}
For $r < 9$, we can improve hypothesis (1) in Theorem
\ref{maincorollary}.  We needed the condition $d > \frac{95}{32}m$
in order to ensure that $L$ to meet all the exceptional
curves $eH-\sum_{i=1}^r n_iE_i$ with $0 \le n_i \le 11$ positively. For
$r\le 8$, we have fewer such
exceptional curves. See the proof of Theorem \ref{main1}.  
As an illustration, when $r=8$, a modification of Theorem
\ref{maincorollary} along these lines requires only $d > 2.66m$ as hypothesis
(1), while hypothesis (2) is unchanged. 
\end{remark}

\begin{example}\label{imp1}
We re-visit Example \ref{imp}: $r=8, m=60$. We already found that
$L_{178}$ is ample using Theorem \ref{main}. We can conclude now that $L_{172}$ is
ample using Theorem \ref{maincorollary} and Remark \ref{analysis}. Note that 
$60\sqrt{\frac{(8)(64)}{63}} = 171.04$ and $172 > (2.66)(60) = 159.6$.

In fact, it turns out that $L_{171}$ is ample. Let $F =
17H-6\sum_{i=1}^8 E_i $. Then $L_{170} = 10F$. There is an element $w
\in W_8$ such that $wF = H$. See Discussion \ref{exccurves}.
This is easy to see by applying the linear map $\gamma \in W_8$
successively to $F$ and
permuting $E_i$ so that their coefficients are non-increasing.
Since $W_8$ preserves intersections and $H$ is nef, $F$ is nef as
well. It now follows that $L_{171} = 10F+H$ is ample: let $C =
dH-\sum_{i=1}^8 n_iE_i$ be an
irreducible, reduced curve. Since $H$ is nef and 
$H\cdot C = d$, it follows that  $d \ge 0$.
If $d > 0$, then 
$L_{171} \cdot C \ge H \cdot C> 0$. Finally suppose that $d=0$. 
If $n_i < 0$ for some $i$ then $C \cdot E_i = n_i < 0$. This implies
that $C = E_i$ and so $L_{171}\cdot C > 0$. On the other hand, if
$n_i \ge 0$ for all $i$ then $C$ can not be effective as it is the
negative of an effective curve. 
\end{example}

\begin{example}\label{not-optimal1}
We re-visit Example \ref{not-optimal} to show that 
$L_7 = 7H-3E_1 - \sum_{i=2}^{10} 2E_i -\sum_{i=11}^{12}
  E_i$ is ample if
the SHGH Conjecture is true. To prove this, first consider  
an irreducible and reduced curve $C$ on 
$X$ with $C^2 <
0$. 
By the SHGH Conjecture, $C$ is a $(-1)$-curve.
Then as we noted in the proof of Lemma \ref{bound-on-e}, there
exists $w \in W_{12}$ such that $C = wE_{12}$.  In fact, one can write 
$C = \sum_{i=0}^{r-1} a_i s_i + E_{12}$ where $s_i$ are simple roots and
$a _i \ge 0$; see \cite[Proposition 1.11]{L1}. 
It is easy to see that $L\cdot s_i \ge 0$ and hence we have
$L \cdot
C \ge L \cdot E_{12} > 0$.

Now let $C$ be an irreducible, reduced curve with $C^2 \ge 0$. Since  
$L_7^2 > 0$, a sufficiently large multiple of $L_7$ is effective by
Riemann-Roch. Thus $L_7 \cdot C \ge 0$. If $L_7 \cdot C = 0$, then
Hodge Index Theorem gives $C^2 < 0$ contradicting the assumption on
$C$. 
\end{example}

Contrary to the situation when $r < 9$, the necessary
condition $L^2 >0$ is also conjectured to be sufficient for ampleness
in the uniform case when $r \ge 9$. Recall the well-known
Nagata Conjecture \cite{N1}: 

\begin{conjecture*} 
Let $p_1,\ldots,p_r$ be general points of $\proj^2$ with $r \ge
9$. Let $n_1,\ldots,n_r$ be non-negative integers.  
Let $C$ be a curve of degree $e$ in $\proj^2$ passing through $p_i$
with multiplicity at least $n_i$. Then $$e \ge \frac{1}{\sqrt{r}}
\sum_{i=1}^r n_i,$$and the inequality is strict for $r \ge
10$. 
\end{conjecture*}

The Nagata Conjecture clearly implies (when $r \ge 9$) 
that $L = dH-m\sum_{i=1}^r E_i$ is ample if
$L^2 > 0$. Indeed, since $L^2 > 0$ we have $d^2 > rm^2$. Now if $C$ is 
an irreducible and reduced
curve on $X$ in the 
linear system of
$eH-\sum_{i=1}^r n_iE_i$, then $e \ge \frac{1}{\sqrt{r}}
\sum_{i=1}^r n_i$ by the Nagata Conjecture. Hence
$de > m\sum_{i=1}^r n_i$ and $L\cdot C >0$.

The conjecture is open except when $r$ is a square in which case
Nagata proved it. Some recent work on this conjecture can be found in
\cite{H1,H2,Ro}. For a nice survey, see \cite{SS}. 

\begin{example}\label{drm-ample}  Let $L_{d,r,m} = dH - m\sum_{i=1}^r E_i$.
\begin{enumerate}
\item First consider $r= m=10$. Then $L_{d,10,10}^2 = d^2 - 1000$. The
  Nagata
Conjecture predicts that $L_{32,10,10}$ is ample.  
Using Theorem \ref{maincorollary} we can verify this. 
Since $10\sqrt{\frac{(10)(70)}{69}} =
31.85$, $L_{32,10,10}$ is ample by Theorem \ref{maincorollary}. 

\item Let $r=10,m=30$. Then $L_{d,10,30}^2 = d^2 - 9000$. The Nagata
Conjecture predicts that $L_{95,10,30}$ is ample. 
Using Theorem \ref{maincorollary} we see that $L_{96,10,30}$ is ample 
since $30\sqrt{\frac{(10)(70)}{69}} =
95.55$.

\item Let $r=30,m=10$. Then $L_{d,30,10}^2 = d^2 - 3000$. The Nagata
Conjecture predicts that $L_{55,10,30}$ is ample. 
Using Theorem \ref{maincorollary} we can indeed verify that $L_{55,30,10}$ is ample 
since $10\sqrt{\frac{(30)(130)}{129}} =
54.98$.
\end{enumerate}

\cite[Theorem 3]{ST} gives similar conditions for ampleness in the
  uniform case. It says that $L_{d,r,m}$ is ample if $d \ge 3m+1$ and
  $d^2 \ge (r+1)m^2$. In Example \ref{drm-ample}, it gives ampleness of
  $L_{34,10,10}$, $L_{100,10,30}$, and $L_{56,30,10}$ in (1),(2) and
  (3) respectively. Our bound will
  always be at least as good as this because $\frac{(3r+40)r}{3r+39} <
  r+1$.  
\end{example}

\begin{remark}
When $m=1$, \cite[Corollary]{K} and \cite[Theorem]{X2} prove the
optimal result about ampleness predicted by the Nagata Conjecture: $L$ is
ample if and only if $L^2 > 0$. 
Our Theorem \ref{maincorollary}, as well as \cite[Theorem
3]{ST}, recover this result.
\end{remark}

\section{Global generation and very ampleness}\label{bpf-va}

In this section, using similar methods as above, we obtain conditions for global
generation and 
very ampleness of $L$ applying Reider's criterion \cite{Re}.

Let $X$ be a smooth surface and let $N$ be a nef line bundle on $X$.
Reider's theorem says that if $N^2 \ge 5$ and $K_X+N$ fails to be globally generated,
there exists an effective divisor $D$ such that 

$D\cdot N = 0, D^2 =
-1$, or 

$D\cdot N = 1, D^2 = 0$. 

Similarly, if $N^2 \ge 10$ and $K_X+N$ fails to be very ample,
there exists an effective divisor $D$ such that 

$D\cdot N = 0, D^2 =
-1{\rm ~or} -2$, or 

$D\cdot N = 1, D^2 = 0{~\rm or -1}$, or 

$D\cdot N =2, D^2=0$. 


We apply Reider's theorem  to $L = dH - \sum_{i=1}^r m_iE_i$ to obtain
conditions for global generation and very ampleness. 

First we consider conditions for global generation. 
\begin{theorem}\label{bpf}
Let $p_1,p_2,\ldots,p_r$ be general  points
  of $\mathbb{P}^2$. Consider the blow up $X \to \proj^2$ of
  $\proj^2$ at the points $p_1,\ldots,p_r$. We denote by $H$ the
  pull-back of a line in $\proj^2$ not passing through any of the
  points $p_i$ and by $E_1,\ldots,E_r$ the
  exceptional divisors. 
Let $L = dH - \sum_{i=1}^{r} m_iE_i$. Suppose that $r \ge 5$ and $m_i
\ge 2$ for all $i$. Then $L$ is globally generated if the following
conditions hold:
\begin{enumerate}
\item $d+3 > m_1+m_2+2$,
\item $2(d+3) > m_1+\ldots+m_5+5$,
\item $3(d+3) > 2m_1+m_2+\ldots+m_7+8$, and
\item ${(d+3)}^2 \ge \frac{s+3}{s+2}\sum_{i=1}^s {(m_i+1)}^2$, for $2
  \le s \le r$.
\end{enumerate}
\end{theorem}
\begin{proof}
Let $N = (d+3)H -\sum_{i=1}^r (m_i+1)E_i$, so that $L = K_X + N$. 
Using our hypotheses we can apply Theorem \ref{main} and
conclude that $N$ is ample. Further 
$N^2 = {(d+3)}^2 - \sum_{i=1}^r {(m_i+1)}^2 \ge  \frac{1}{r+2}
\sum_{i=1}^r {(m_i+1)}^2$. Since $r \ge 5$ and $m_i \ge 2$ for
all $i$, it follows that  
$N^2 \ge 5$. Thus we can apply Reider's theorem. 
Note that since $N$ is ample, there is no effective divisor $D$ such
that $D\cdot N = 0$. 
Hence if $L$ is
not globally generated there exists an
effective divisor $D$ such that $D^2 = 0$ and $D \cdot N = 1$.

Suppose $D = eH -
\sum_{i=1}^r n_i E_i$ with $e > 0$ and $n_i$ non-negative
integers. Without loss of generality, assume that $n_i > 0$ for all
$i$. 
Then we have
$e^2 = \sum_{i=1}^r n_i^2$ and $(d+3)e = \sum_{i=1}^r (m_i+1)n_i + 1$. 
For simplicity, let $a = \sum_{i=1}^r {(m_i+1)}^2$, $b = \sum_{i=1}^r
n_i^2$ and $c = \sum_{i=1}^r (m_i+1)n_i$.

We show that $(d+3)e > c+1$, which contradicts $D\cdot
N = 1$.  

We have  ${(d+3)}^2e^2 > \frac{r+3}{r+2}ab$. We proceed as in the proof
of  
Lemma \ref{key} to show that $\frac{r+3}{r+2}ab \ge {(c+1)}^2$. This
inequality is equivalent to $ab+(r+2)(ab-c^2) \ge (r+2)(2c+1)$. Since
$ab-c^2 \ge 0$, this follows if $ab \ge (r+2)(2c+1)$. 
Using the inequality $ab \ge c^2$ again, it suffices to prove that
$c^2 \ge (r+2)(2c+1)$. It is clear that the least value of
$c^2-(r+2)(2c+1)$ is attained when $m_i+1$ and $n_i$ take the least
values allowed. By hypothesis, $m_i \ge 2$ and $n_i \ge 1$. So the
required inequality in this case is $9r^2 \ge (r+2)(6r+1)$. This holds
for $r \ge 5$.
\end{proof}

\begin{example} As in Example \ref{optimal}, consider 
$L_d = dH-3E_1 - \sum_{i=2}^82E_i -\sum_{i=9}^{12}
  E_i$. Then the smallest $d$ which satisfies the hypotheses (1)-(4)
  of Theorem \ref{bpf} is $d = 8$: $(8+3)^2 > \frac{15}{14}(99) =
  106.07$. It is clear that $d=8$ satisfies the other hypotheses in
  Theorem \ref{bpf}.  
 So  $L_8$ is base point free. We saw that $L_7$ was ample in Example
 \ref{optimal}. 
\end{example}

\begin{example} As in Example \ref{not-optimal}, consider
$L_d = dH-3E_1 - \sum_{i=2}^{10}2E_i -\sum_{i=11}^{12}
  E_i$. Then the smallest $d$ which satisfies the hypotheses (1)-(4)
  of Theorem \ref{bpf} is $d = 8$. By hypothesis (4), we require $(d+3)^2
  \ge \frac{15}{14}(105) = 112.5$. It is easy to see that $d=8$
  satisfies all the other hypotheses. 
So $L_8 = 8H-3E_1 - \sum_{i=2}^{10}2E_i -\sum_{i=11}^{12}
  E_i$ is base point free. As we saw in Example \ref{not-optimal}, $L_8$ is also ample.
\end{example}

\begin{example}
As in Example \ref{tight}, consider 
$L_d = dH - 10 \sum_{i=1}^5 E_i$. 
By Theorem \ref{bpf},
$L_{25}$ is globally generated. It is easy to see that the required
conditions are satisfied.
\end{example}

Our next result gives better bounds for global generation in the uniform
case when $m \ge 6$. 

\begin{theorem}\label{bpf-uniform}
Let $p_1,p_2,\ldots,p_r$ be general  points
  of $\mathbb{P}^2$. Consider the blow up $X \to \proj^2$ of
  $\proj^2$ at the points $p_1,\ldots,p_r$. We denote by $H$ the
  pull-back of a line in $\proj^2$ not passing through any of the
  points $p_i$ and by $E_1,\ldots,E_r$ the
  exceptional divisors. 
Let $L = dH - m \sum_{i=1}^{r} E_i$. Suppose that $r \ge 2$ and $m
\ge 6$. Then $L$ is globally generated if 
the following conditions hold:
\begin{enumerate}
\item $d \ge \frac{95}{32}m$, and
\item $(d+3)^2 \ge \big{(}\frac{3r+5}{3r+4}\big{)}r(m+1)^2$.
\end{enumerate}
\end{theorem}

\begin{proof}
As in the proof of Theorem \ref{bpf}, 
let $N = (d+3)H -(m+1) \sum_{i=1}^r E_i$, so that $L = K_X + N$. 
Using our hypotheses we can apply Theorem \ref{maincorollary} and
conclude that $N$ is ample. 
Note that $\frac{3r+5}{3r+4} > \frac{3r+40}{3r+29}$. 
Further 
$N^2 = {(d+3)}^2 - r(m+1)^2 \ge  \frac{r(m+1)^2}{3r+4}$. Since $m \ge
6$, it is clear that 
$N^2 \ge 5$. We can thus apply Reider's theorem. 
If $L$ is
not globally generated there exists an
effective divisor $D$ such that $D^2 = 0$ and $D \cdot N = 1$. 

Suppose $D = eH -
\sum_{i=1}^r n_i E_i$ with $e > 0$ and $n_i$ non-negative
integers. Without loss of generality, assume that $n_i > 0$ for all
$i$. Write $n_1 \ge n_2 \ge \ldots \ge n_r > 0$. 
Let $b = \sum_{i=1}^r n_i^2$ and $c = \sum_{i=1}^r n_i$. Then $e^2 =
b$. 

By hypothesis, $(d+3)^2e^2 \ge \big{(}\frac{3r+5}{3r+4}\big{)}(m+1)^2
rb$. Note that $D \cdot N = (d+3)e-(m+1)c$. We claim now that 
$\big{(}\frac{3r+5}{3r+4}\big{)}(m+1)^2
rb > \big{(} (m+1)c+1\big{)}^2$, which contradicts $D \cdot N = 1$. 

The required inequality is equivalent to 

$(3r+5)(m+1)^2 rb
> (3r+4)\big{(}(m+1)^2c^2+2(m+1)c+1\big{)}$

$\Leftrightarrow (m+1)^2rb + (3r+4)(m+1)^2(rb-c^2) >
(3r+4)(2(m+1)c+1)$. 

Since $rb-c^2 \ge 0$, it suffices to show that
\begin{eqnarray}\label{3.5}
(m+1)^2rb  >(3r+4)(2(m+1)c+1)= 6(m+1)rc+8(m+1)c+3r+4.
\end{eqnarray}

First suppose that $n_1 = 1$. Then we have $e^2 = r$. 
By a dimension count, we see that for $e \ge 4$, such curves don't
exist. So only such curves are $D=H-E_1$,
$D = 2H-E_1-\ldots-E_4$ and $D = 3H-E_1-\ldots-E_9$. 
Clearly $N \cdot (H-E_1) > 1$. So assume that $e \ge 2$. 
We have $N^2 > 0$ by hypothesis (2). So $(d+3)^2>(m+1)^2r
\Longrightarrow d+3 > (m+1)e \Longrightarrow (d+3)-(m+1)e \ge 1
\Longrightarrow (d+3)e-(m+1)r \ge e \ge 2$. Hence 
$N \cdot D > 1$.

Now suppose that $n_1 \ge 2$. We re-write \eqref{3.5} as follows. 
\begin{eqnarray}\label{3.6}
b  > c\left(\frac{6}{m+1}+\frac{8}{r(m+1)}\right)+\frac{3r+4}{(m+1)^2r}.
\end{eqnarray}

Note that $b=\sum n_i^2$ grows faster than
$c = \sum n_i$. For fixed $r$ and $m$, if \eqref{3.6} holds for
$(n_1,n_2,\ldots,n_r)$ then it holds for
$(n_1,\ldots,n_i+1,\ldots,n_r)$. Thus it suffices to check \eqref{3.6}
for  $n_1=2, n_2 = \ldots=n_r=1$. 
In this case,
$b = r+3$ and $c = r+1$. 
So \eqref{3.6} is equivalent to 
$$(r+3)  >(r+1)
\left(\frac{6}{m+1}+\frac{8}{r(m+1)}\right)+\frac{3r+4}{(m+1)^2r} = 
\frac{r+1}{m+1}\left(6+\frac{8}{r}\right)+\frac{3r+4}{(m+1)^2r}.$$
It is easy to see that this holds when 
$m \ge 6$ and $r\ge 2$. 
\end{proof}

\begin{example}\label{drm-bpf}
As in Example \ref{drm-ample}, let 
$L_{d,r,m} = dH - m\sum_{i=1}^r E_i$.
\begin{enumerate}
\item $r= m=10$. 
By Theorem \ref{bpf-uniform}, $L_{33,10,10}$ is base point free
since $11\sqrt{\frac{(10)(35)}{34}} =
35.29$. Recall that $L_{32,10,10}$ is ample by Theorem \ref{maincorollary}. 

\item $r=10,m=30$. 
Since $31\sqrt{\frac{(10)(35)}{34}} =
99.46$, $L_{97,10,30}$ is base point free. Recall that
$L_{96,10,30}$ is ample. 

\item Let $r=30,m=10$. 
$L_{58,30,10}$ is base point free
since $11\sqrt{\frac{(30)(95)}{94}} =
60.56$. Recall that $L_{55,30,10}$ is ample. 
\end{enumerate}

\cite[Theorem 4]{ST} gives similar conditions for global generation in the
  uniform case. It says that $L_{d,r,m}$ is globally generated if $d \ge 3m+1$ and
  $(d+3)^2 \ge (r+1)(m+1)^2$. In this example, it gives global generation of
  $L_{34,10,10}$, $L_{100,10,30}$, and $L_{59,30,10}$. Our bound will
  always be at least as good as this because $\frac{(3r+5)r}{3r+4} <
  r+1$.
\end{example}

We now give conditions for very ampleness of a line bundle in the
uniform case. 

\begin{theorem}\label{veryample}
Let $p_1,p_2,\ldots,p_r$ be general  points
  of $\mathbb{P}^2$. Consider the blow up $X \to \proj^2$ of
  $\proj^2$ at the points $p_1,\ldots,p_r$. We denote by $H$ the
  pull-back of a line in $\proj^2$ not passing through any of the
  points $p_i$ and by $E_1,\ldots,E_r$ the
  exceptional divisors. 
Let $L = dH - m \sum_{i=1}^{r}E_i$. Suppose that $r \ge 3$ and $m
\ge 4$. 
Then $L$ is very ample if the
following conditions hold:
\begin{enumerate}
\item $d \ge 3m$, and 
\item ${(d+3)}^2 \ge \big{(}\frac{r+3}{r+2}\big{)}r{(m+1)}^2$.
\end{enumerate}
\end{theorem}
\begin{proof}

Let $N =
(d+3)H -(m+1) \sum_{i=1}^ r E_i$. Then $N$ is ample by
Theorem \ref{maincorollary}. Further, using hypothesis (2), we get 
$N^2 = {(d+3)}^2-r{(m+1)}^2 \ge \frac{r}{r+2}{(m+1)}^2 $. Since $r \ge
3$ and $m \ge 4$,  we obtain
$N^2\ge 10$. So we can apply Reider's theorem. 

Suppose that $L$ is not very ample. By Reider's theorem there
exists an effective divisor $D$ such that 

$D\cdot N = 1, D^2 = 0{~\rm or -1}$, or 

$D\cdot N =2, D^2=0$. 

We rule out the case $D\cdot N = 1, D^2 = 0$ exactly as in the proof
of Theorem
\ref{bpf-uniform}. 

Let $D$ be an effective divisor such that $D\cdot N = 2, D^2 =
0$. As in the proof of Theorem \ref{bpf-uniform}, 
write $D = eH -
\sum_{i=1}^r n_i E_i$ with $e > 0$ and $n_i$ non-negative
integers. Without loss of generality, assume that $n_i > 0$ for all
$i$. Write $n_1 \ge n_2 \ge \ldots \ge n_r > 0$. 
Let $b = \sum_{i=1}^r n_i^2$ and $c = \sum_{i=1}^r n_i$. Then $e^2 =
b$. 

By hypothesis, $(d+3)^2e^2 \ge \big{(}\frac{r+3}{r+2}\big{)}(m+1)^2
rb$. Note that $D \cdot N =  (d+3)e-(m+1)c$.  
We claim now that 
$\big{(}\frac{r+3}{r+2}\big{)}(m+1)^2
rb > ((m+1)c+2)^2$, which contradicts $D \cdot N = 2$. 

The required inequality is equivalent to 

$(r+3)(m+1)^2 rb
> (r+2)\big{(}(m+1)^2c^2+4(m+1)c+4\big{)}$

$\Leftrightarrow (m+1)^2rb + (r+2)(m+1)^2(rb-c^2) >
(r+2)(4(m+1)c+4)$. 

Since $rb-c^2 \ge 0$, it suffices to show that
\begin{eqnarray}\label{3.7}
(m+1)^2rb  >(r+2)(4(m+1)c+4)= 4(m+1)rc+8(m+1)c+4r+8.
\end{eqnarray}

If $n_1 = 1$ we argue exactly as in the proof of Theorem
\ref{bpf-uniform}. So suppose that $n_1 \ge 2$. 

We re-write \eqref{3.7} as follows. 
\begin{eqnarray}\label{3.8}
b  > c\left(\frac{4}{m+1}+\frac{8}{r(m+1)}\right)+\frac{4r+8}{(m+1)^2r}.
\end{eqnarray}

Note that $b=\sum n_i^2$ grows faster than
$c = \sum n_i$. For fixed $r$ and $m$, if \eqref{3.8} holds for
$(n_1,n_2,\ldots,n_r)$ then it holds for
$(n_1,\ldots,n_i+1,\ldots,n_r)$. Thus it suffices to check \eqref{3.8}
for  $n_1=2, n_2 = \ldots=n_r=1$. 
In this case,
$b = r+3$ and $c = r+1$. 
So \eqref{3.8} is equivalent to 
$$(r+3)  >(r+1)
\left(\frac{4}{m+1}+\frac{8}{r(m+1)}\right)+\frac{4r+8}{(m+1)^2r} = 
\frac{r+1}{m+1}\left(4+\frac{8}{r}\right) +\frac{4r+8}{(m+1)^2r}.$$
It is easy to see that this holds when 
$m \ge 4$ and $r\ge 3$.

Now we consider the case $D\cdot N = 1, D^2 = -1$. 
As above, write $D = eH -
\sum_{i=1}^r n_i E_i$ with $e > 0$ and $n_1 \ge n_2 \ge \ldots \ge n_r > 0$. 

Since $D\cdot N = 1$ and $N$ is ample, we may assume that $D$ is an
irreducible and reduced curve. 

Suppose that $n_1 \le 11$. By \cite[Theorem 34]{DJ}, any
irreducible, reduced curve of negative self-intersection and 
with multiplicities of 11 or less at $p_i$ is a $(-1)$-curve.
So $D$ is a $(-1)$-curve and hence  $D \cdot K_X =
-1$. This means that  $3e = 1 + \sum_i n_i$. 
As $d \ge 3m$, we have $d+3 \ge 3(m+1)$ and this implies that 
$(d+3)e \ge (m+1)(1+\sum_i n_i) =
(m+1)+(m+1)\sum_i n_i$, 
so $D \cdot N \ge m+1 > 1$ which contradicts the hypothesis that $D
\cdot N = 1$.   

Next suppose that $n_1 \ge 12$. 
By Lemma \ref{key1} below, 
$\frac{r(r+3)}{r+2} e^2 > {\big{(}\sum_{i=1}^rn_i +
  \frac{1}{2}\big{)}}^2$. Thus $e^2 > \frac{r+2}{r(r+3)} {\big{(}\sum_{i=1}^rn_i +
  \frac{1}{2}\big{)}}^2$. 
By hypothesis (2), ${(d+3)}^2 \ge  \frac{r(r+3)}{r+2}{(m+1)}^2$. 

So
 ${(d+3)}^2e^2  >  {(m+1)}^2{\big{(}\sum_{i=1}^rn_i +
  \frac{1}{2}\big{)}}^2$, giving 
$(d+3)e > (m+1) \big{(}\sum_{i=1}^rn_i +
  \frac{1}{2}\big{)}$.
Then $(d+3)e - (m+1) \big{(}\sum_{i=1}^rn_i \big{)} >  \frac{m+1}{2}
\ge 2$, again contradicting the hypothesis that 
$D \cdot N =1$.

We conclude that $L$ is very ample. 
\end{proof}

Theorem \ref{veryample} gives conditions for very ampleness for a 
line bundle of the form  $dH - m \sum_{i=1}^{r}E_i$ if $m
\ge 4$ and $r \ge 3$. The case $m=1$ is well-studied. In fact, a
complete characterization is known in this case, see \cite[Proposition
3.2]{GGP}.

\begin{example}\label{drm-va} 
As in Examples \ref{drm-ample} and \ref{drm-bpf}, let 
$L_{d,r,m} = dH - m\sum_{i=1}^r E_i$.
\begin{enumerate}
\item $r= m=10$. 
By Theorem \ref{veryample}, $L_{34,10,10}$ is very ample 
since $11\sqrt{\frac{(10)(13)}{12}} =
36.20$. Recall that $L_{32,10,10}$ is ample 
and $L_{33,10,10}$ is base point free.

\item $r=10,m=30$. 
Since $31\sqrt{\frac{(10)(13)}{12}} =
102.03$, $L_{100,10,30}$ is very ample. Recall that
$L_{96,10,30}$ is ample and $L_{97,10,30}$ is base point free. 

\item Let $r=30,m=10$. 
$L_{59,30,10}$ is very ample
since $11\sqrt{\frac{(30)(33)}{32}} =
61.18$. Recall that $L_{55,30,10}$ is ample  and 
$L_{58,30,10}$ is base point free. 
\end{enumerate}
\end{example}

Finally we prove the following lemma which was used in the proof of
Theorem \ref{veryample}.
 
\begin{lemma}\label{key1}
Let $n_1 \ge n_2 \ge \ldots \ge n_r > 0$ be positive
integers with $r \ge 2$ and $n_1 \ge 12$. 
Then $\frac{(r+3)r}{r+2}(n_1^2+n_2^2+\ldots+n_r^2-1) >
{(n_1+n_2+\ldots+n_r+\frac{1}{2})}^2$.
\end{lemma}
\begin{proof}

Let $a = n_1^2+n_2^2+\ldots+n_r^2$, $b = n_1+n_2+\ldots+n_r$ and
$c = \sum_{i < j} {(n_i-n_j)}^2$. Note that $ra-b^2 = c$. 

We first consider the uniform case: $n_1=n_2=\ldots n_r = n$. Then $a = rn^2,
b= rn$. So the desired inequality is 
$\frac{(r+3)r}{r+2}(rn^2-1) >
r^2n^2+rn+\frac{1}{4}$. It is equivalent to 
\begin{eqnarray*}
r^3n^2+3r^2n^2-r^2-3r &>& r^3n^2+2r^2n^2+r^2n+2rn+r/4+1/2\\
\Longleftrightarrow r^2(n^2-n-1) &>& (2n+3.25)r+1/2
\end{eqnarray*}
This clearly holds for $r \ge 2$, $n \ge 12$.

If all $n_i$ are not all equal to each other, then 
$ra-b^2 = c \ge r-1$. Substituting $ra= b^2+c$, the desired
inequality is  
$(r+3)(b^2+c-r) > (r+2)(b^2+b+1/4)$. Since $c-r \ge -1$, it 
suffices to show that 
$(r+3)(b^2-1) >  (r+2)(b^2+b+1/4)$.
This is equivalent to 
\begin{eqnarray}\label{last} 
b^2 > b(r+2)+r+r/4+3.5.
\end{eqnarray}

For a fixed $r$, it is easy to see that if \eqref{last} holds for some
$b=\sum_i^r n_i$, then it holds for $b+1$ also. 
Therefore it suffices to check \eqref{last} for the smallest value of $b$. This is
attained when $(n_1,n_2,\ldots,n_r) = (12,1,\ldots,1)$. In this case,
$b = r+143$ and  \eqref{last} holds in this case, as one can see by an easy calculation.
\end{proof}

\bibliographystyle{plain}

\end{document}